\documentclass[EJP]{ejpecp_modify} 

\SHORTTITLE{Transformations of Wiener Measure and Orthogonal Expansions} 

\TITLE{Transformations of Wiener Measure and Orthogonal Expansions} 

\AUTHORS{%
  Andrey~A.~Dorogovtsev\footnote{Institute of Mathematics, NAS of Ukraine. \BEMAIL{adoro@imath.kiev.ua, ryabov.george@gmail.com}}
  \and 
  Georgii~Riabov\footnotemark[1]}

\KEYWORDS{Wiener measure; chaos expansion; coalescence; stochastic flows} 

\AMSSUBJ{60H07; 60H30; 60K35} 

\SUBMITTED{} 
\ACCEPTED{} 

\VOLUME{0}
\YEAR{2013}
\PAPERNUM{0}
\DOI{vVOL-PID}

\ABSTRACT{In this paper we study the structure of square integrable functionals measurable with respect to coalescing stochastic flows. The case of $L^2$ space generated by the process $\eta(\cdot)=w(\min(\tau,\cdot)),$ where $w$ is a Brownian motion and $\tau$ is the first moment when $w$ hits the given continuous function $g$ is considered. We present a new construction of multiple stochastic integrals with respect to the process $\eta.$ Our approach is based on the change of measure technique. The analogue of the It\^o-Wiener expansion for the space $L^2(\eta)$ is constructed.}

\newcommand{\vf}{\varphi}
\newcommand{\vk}{\varkappa}

\newcommand{\cB}{{\mathcal B}}
\newcommand{\cK}{{\mathcal K}}
\newcommand{\1}{1\!\!\,{\rm I}}
\newcommand{\mbR}{{\mathbb R}}

\begin{document}



\section{Introduction}

The main objective of the research we undertake in the present paper is to get a description of the square integrable functionals measurable with respect to certain coalescing stochastic flows. Our motivation for such investigation is to gain better understanding of the noise associated with the coalescing stochastic flow. The description of square integrable functionals measurable with respect to the smooth stochastic flow is well-known. Namely, consider a stochastic flow $(\vf_{s,t})_{0\leq s\leq t}$ on $\mbR$ generated by the stochastic differential equation
$$
\begin{cases}
d\vf_{s,t}(x)=a(t,\vf_{s,t}(x))dt+b(t,\vf_{s,t}(x))dw(t) \\
\vf_{s,s}(x)=x.
\end{cases}
$$
with bounded and measurable coefficients $a$ and $b,$ $b\geq \delta>0.$ Then every square integrable $\vf-$measurable functional admits a unique expansion as a series of multiple stochastic integrals with respect to initial Wiener process $w$ (the so-called It\^o-Wiener expansion). The well-known Krylov-Veretennikov formula \cite{VK} is an example of such expansion for a special functional $f(\vf_{s,t}(x))$
\begin{equation}
\label{intro2}
f(\vf_{s,t}(x))=\sum^{\infty}_{n=0} \int_{\Delta_n(s, t)}\Bigg( T_{s,t_1}b\frac{\partial}{\partial y}T_{t_1,t_2}
\ldots b\frac{\partial}{\partial y} T_{t_n,t} f\Bigg)(x)dw(t_1)\ldots dw(t_n).
\end{equation}
Here $f$ is bounded and measurable function on $\mbR$ and $T_{s,t}g(y)$ is the solution in $(s,y)$ of the parabolic boundary value problem
$$
\begin{cases}
u'_s+\frac{1}{2}b(s,y)^2 u''_{y y}+a(s,y) u'_y=0 \\
u(t,y)=g(y),
\end{cases}
$$
$\Delta_n(s,t)$ is the $n-$dimensional simplex $\{(t_1,\ldots , t_n): \ s<t_1<\ldots<t_n<t )\}.$

Trying to obtain an analogue of the It\^o-Wiener expansion for func\-ti\-o\-nals from a coalescing stochastic flow one 
faces the difficulty that in ge\-ne\-ral there is no Gaussian noise that generates the flow \cite{lJR}. The Arratia flow $\{x(u,t)\}_{u\in \mbR, t\geq 0}$ of 
coalescing Brownian particles on the line \cite{Arr} is the typical example of such situation. Still, finite set $\{x(u_1,\cdot), \ldots , x(u_n,\cdot)\}$ of 
trajectories from Arratia flow can be built from independent Wiener processes $\{\tilde{w}_1,\ldots ,\tilde{w}_n\}$ and a functional from trajectories $x(u_i,\cdot),$ $1\leq i\leq n$ can be expressed as a functional from $\tilde{w}_i,$ $1\leq i\leq n.$ Such approach was developed in \cite{D} to obtain the Krylov-Veretennikov expansion for the 
functional $f(x(u_1,t), \ldots , x(u_n,t))$ in terms of stochastic integrals with respect to Wiener processes $\tilde{w}_i'$s. However, the obtained expansion is not intrinsic in the sense it depends on the way trajectories $x(u_1,\cdot),\ldots ,x(u_n,\cdot)$ are constructed from Wiener processes $\tilde{w}_1,\ldots \tilde{w}_n.$          

The present article is devoted to the detailed study of a model example of a coalescing flow. We describe the structure of square integrable functionals measurable with respect to the coalescing flow of two particles, a Brownian one and a ``heavy'' deterministic one. That is, after a moment of meeting both particles move together according to the law of the deterministic one. Formally, let $(w(t))_{t\in [0,1]}$ be a Wiener process on a probability space $(\Omega, \mathcal{F}, \mathbb{P})$ and $g:[0,1]\to \mbR$ be a given continuous function. Let $\tau_g$ be the first moment when $w$ hits the continuous function $g.$ Then the noise generated by the described system is simply the noise generated by the stopped process $\eta_g=w(\min(\tau_g,\cdot)).$ In what follows we search for a description of $\eta_g-$measurable functionals as a series of stochastic integrals with respect to $\eta_g.$ Evidently, multiple stochastic integrals
\begin{equation}
\label{eq00}
\int^{\tau_g}_0\int^{t_n}_0\ldots \int^{t_2}_0 a_n(t_1,\ldots,t_n) d\eta_g(t_1)\ldots d\eta_g(t_n), \ a_n\in L^2_{symm}([0,1]^n).
\end{equation}
form a total subset of $L^2(\eta_g).$ However, two integrals of a kind \eqref{eq00} having different order may be not orthogonal. Moreover, an expansion of the function $f\in L^2(\eta_g)$ into the sum
$$
f=\sum^{\infty}_{n=0} \int^{\tau_g}_0\int^{t_n}_0\ldots \int^{t_2}_0 a_n(t_1,\ldots,t_n) d\eta_g(t_1)\ldots d\eta_g(t_n)
$$
is not unique, as the following example shows.

\begin{example}
Consider the function $f=w(\tau_g)=\eta_g(1).$ Evidently,
\begin{equation}
\label{ex01}
f=\int^{\tau_g}_0 d\eta_g(t).
\end{equation}
On the other hand, $f$ is measurable with respect to $w$ and it admits the It\^o-Wiener expansion
$$
f=\sum^{\infty}_{n=1} \int^{1}_0\int^{t_n}_0\ldots \int^{t_2}_0 a_n(t_1,\ldots,t_n) dw(t_1)\ldots dw(t_n).
$$
Taking the conditional expectation  with respect to $\eta_g$ one obtains another representation
\begin{equation}
\label{ex02}
f=\sum^{\infty}_{n=0} \int^{\tau_g}_0\int^{t_n}_0\ldots \int^{t_2}_0 a_n(t_1,\ldots,t_n) d\eta_g(t_1)\ldots d\eta_g(t_n).
\end{equation}
For every $b\in L^2[0,1]$ one has
$$
\int^1_0 a_1(t) b(t) dt= \mathbb{E}f\int^1_0 b(t)dw(t)= \int^1_0 b(t)\mathbb{P}(\tau_g\geq t) dt.
$$
Hence, $a_1(t)=\mathbb{P}(\tau_g\geq t)$ and
$$
\int^{\tau_g}_0 d\eta_g(t) \ne \int^{\tau_g}_0 a_1(t)d\eta_g(t).
$$
This proves that the representations \eqref{ex01} and \eqref{ex02} are different.
\end{example}

Applications of several orthogonalisation procedures to integrals \eqref{eq00} were described in \cite{4,2}. Still, the detailed description of the resulting orthogonal objects remained as an open problem.

The approach we propose in the present article is based on transformations of the initial probability measure $\mathbb{P}$ by functionals $\1_{\tau_g\geq t}$ in the manner of \cite[ch. 9]{5}. Namely, we study the structure of functionals from $w$ that are square integrable with respect to the new probability measure $d\tilde{\mathbb{P}^t}=\frac{\1_{\tau_g \geq t}}{\mathbb{P}(\tau_g \geq t)}d\mathbb{P}.$ Such approach allows to use the Girsanov theorem, which leads to combinations of stochastic integrals of the kind
$$
J^t_na_n=\sum^{n}_{m=0}\int_{[0,t]^{n-m}}\int_{[0,t]^m}a_n(s,r)k^t_{n,m}(r,w(r))dr dw(s)
$$
with explicitly described kernels $k^t_{n,m}$, such that for $n\ne k$ functions $J^t_na_n$ and $J^t_kb_k$ are orthogonal with respect to the measure $\tilde{\mathbb{P}}^t$ and
each $f\in L^2(\sigma(w),\tilde{\mathbb{P}}^t)$ can be uniquely expanded as a sum $f=\sum^{\infty}_{n=0} J^t_na_n.$ In other words, functions $J^t_na_n$ constitute the orthogonal expansion of the space
$L^2(\sigma(w),\tilde{\mathbb{P}}^t)$ analogous to the It\^o-Wiener expansion. Further we show that stochastic integrals of the form
$\int^{\tau_g}_0 J^t_n a^t_n d\eta_g(t)$ constitute the orthogonal expansion of the space $L^2(\eta_g).$

The article is organized in the following way. A short section 2 is devoted to the preliminary material on the structure of the space of square integrable functionals on the Banach space $\mathcal{X}$ with Gaussian measure $\mu.$ In section 3 we construct the analogue of the It\^o-Wiener expansion of the space $L^2(\mathcal{X},\nu)$ determined by a measurable isomoprhism $T$ between probability spaces $(\mathcal{X},\nu)$ and $(\mathcal{X},\mu).$ It is proved that under mild assumptions on $T$ elements of the constructed expansion can be expressed as a certain compensation of $\mu-$orthogonal polynomials (theorem 3.3). Section 4 is devoted to the case of classical Wiener space. Application of Girsanov theorem to the construction of measurable isomorphisms on the classical Wiener space is discussed. The measurable isomorphism cor\-res\-pon\-ding to the functional $\1_{\tau_g\geq t}$ is constructed (theorem 4.6). As a corollary, an orthogonal expansion of $L^2(\sigma(w),\tilde{\mathbb{P}}^t)$ is built. Finally, in section 5 the orthogonal expansion of $L^2(\eta_g)$ analogous to the It\^o-Wiener expansion of $L^2(w)$ is constructed (theorem 5.2).

\section{Preliminary material}

In this section we collect definitions and constructions necessary for further exposition.

$\mathcal{X}$ is a real separable Banach space; $\cB$ is the Borel $\sigma$-field in $\mathcal{X}$;
$\mu$ is a centered Gaussian measure  on $(\mathcal{X},\cB)$ with $\mbox{supp} \ \mu=\mathcal{X}$;
$\mathcal{H}$ is the Cameron-Martin space of the measure $\mu,$ that is  a real separable Hilbert space densely and compactly embedded into $\mathcal{X}$ and satisfying the property $\int_{\mathcal{X}} e^{il(x)}\mu(dx)=e^{-\frac{1}{2} |l|^2_{\mathcal{H}}},$ $l\in \mathcal{X}^*$ \cite{Bog_G}.

$L^2(\mathcal{X},\mu)$ is the space of all $\mu-$square integrable $\mbR-$valued functions; $\mathcal{L}^{(n)}_{HS}(\mathcal{H})$ is the space of all $n$-linear symmetric Hilbert-Schmidt forms $A_n:\mathcal{H}^n\to \mbR$;
$\mathcal{L}^{(n)}_{HS,fin}(\mathcal{H})$ is the set of all forms $A_n\in \mathcal{L}^{(n)}_{HS}(\mathcal{H})$ satisfying
\begin{equation}
\label{eq1}
A_n(h,\ldots,h)=\prod^d_{i=1} (h,e_i)^{k_i}, \ h\in\mathcal{H}
\end{equation}
for some orthonormal in $\mathcal{H}$ system of elements $e_1,\ldots,e_d \in \mathcal{X}^*$.
The span of $\mathcal{L}^{(n)}_{HS,fin}(\mathcal{H})$ is dense in $\mathcal{L}^{(n)}_{HS}(\mathcal{H}).$

$H_k(t)=(-1)^k e^{\frac{t^2}{2}}(e^{-\frac{t^2}{2}})^{(k)}, \ t\in \mbR$ is the $k$-th Hermite polynomial.

$I^{\mu}_n: \mathcal{L}^{(n)}_{HS}(\mathcal{H}) \rightarrow L^2(\mathcal{X},\mu)$ is the unique isomorphic embedding whose action on the form $A_n\in \mathcal{L}^{(n)}_{HS,fin}(\mathcal{H})$ satisfying \eqref{eq1} is given by the formula
$$
I^{\mu}_n A_n (x)=\prod^d_{i=1} H_{k_i} (e_i(x)), \ x\in \mathcal{X},
$$
(see \cite{1,D_stoch_an,M}). In fact, $\frac{1}{\sqrt{n!}}I^{\mu}_n$ is the isometry between $\mathcal{L}^{(n)}_{HS}(\mathcal{H})$ and its image $\cK_n(\mathcal{X},\mu)=I^{\mu}_n(\mathcal{L}^{(n)}_{HS}(\mathcal{H})).$ The space $\cK_n(\mathcal{X},\mu)$ is exactly the space of $\mu-$orthogonal polynomials of the degree $n.$ Spaces $\cK_n(\mathcal{X},\mu)$ are mutually orthogonal and $L^2(\mathcal{X},\mu)$ coincides with their Hilbert sum (the It\^o-Wiener expansion):
$$
L^2(\mathcal{X},\mu)=\mathop{\oplus}\limits^\infty_{n=0}\cK_n(\mathcal{X},\mu).
$$
Analogous construction is valid in the case of $\mathcal{E}-$valued $\mu-$square integrable functionals, where $\mathcal{E}$ is any real separable Hilbert space \cite{Bog_G}.

$\mathcal{C}_t$ is the space of all continuous functions $x:[0,t]\rightarrow \mbR, \ x(0)=0$ endowed with the sup-norm; $\cB_t$ is the Borel $\sigma-$field in $\mathcal{C}_t,$ $\cB_t$ is identified with the $\sigma-$field
$\sigma\{x(s): \ s\in[0,t]\}$ in $\mathcal{C}_1.$ Index $t$ will be supressed when $t=1.$
The Wiener measure on the space $(\mathcal{C}_t,\cB_t)$ will be denoted by $\mu_t.$ The Cameron-Martin space $\mathcal{H}_t$  of the measure $\mu_t$ is identified with $L^2([0,t])$ via the inclusion  \cite{1}
$$
L^2([0,t])\ni h \mapsto \int^{\cdot}_0 h(s)ds \in \mathcal{C}_t.
$$
The space $\mathcal{L}^{(n)}_{HS}(\mathcal{H}_t)$ is identified with the space $L^2_s([0,t]^n)$ of all symmetric functions from $L^2([0,t]^n).$ The mapping $I^{\mu_t}_n$ coincides with the $n-$fold stochastic integration with respect to the canonical Wiener process $(w_s)_{s\in[0,t]}$ on $(\mathcal{C}_t,\mu_t),$ $w_s(x)=x(s)$ \cite{1}:
$$
I^{\mu_t}_n a_n =\int_{[0,t]^n} a_n(s) dw_s:=n!\int^t_0\int^{s_n}_0\ldots \int^{s_2}_0 a_n(s_1,\ldots,s_n)dw_{s_1}\ldots dw_{s_n}.
$$
The space $(\mathcal{C},\mathcal{H},\mu)$ will be referred to as the classical Wiener space.

\section{Measurable isomorphisms and orthogonal expansions}

In the following simple lemma we show that a measurable and injective mapping $T:\mathcal{X}\to \mathcal{X}$ taking measure $\nu$ to the Gaussian measure $\mu$ allows to transfer the It\^o-Wiener expansion of $L^2(\mathcal{X},\mu)$ to the
space $L^2(\mathcal{X},\nu).$

\begin{lemma}\label{isomorphism}
Let $\nu$ be a probability measure on $(\mathcal{X},\cB).$ Assume that a Borel set
$\mathcal{X}_0\subset \mathcal{X}$ and a mapping $T:\mathcal{X}_0\rightarrow \mathcal{X}$ satisfy conditions

1) $\nu(\mathcal{X}_0)=1$;

2) T is measurable and injective;

3) $\nu \circ T^{-1}=\mu$.

Then the mapping $\widehat{T}f=f\circ T$ is the isometry between spaces $L^2(\mathcal{X},\mu)$ and $L^2(\mathcal{X},\nu).$ In particular, mappings
$$
I^{\nu}_n=\widehat{T}\circ I^{\mu}_n
$$
satisfy following properties:

1) $\frac{1}{\sqrt{n!}}I^{\nu}_n:\mathcal{L}^{(n)}_{HS}(\mathcal{H})\to L^2(\mathcal{X},\nu)$ is the isometrical embedding;

2) spaces $\cK_n(\mathcal{X},\nu)=I^{\nu}_n(\mathcal{L}^{(n)}_{HS}(\mathcal{H}))$ are mutually orthogonal;

3) $L^2(\mathcal{X},\nu)=\mathop{\oplus}\limits^\infty_{n=0}\cK_n(\mathcal{X},\nu).$

Equivalently, for each $f\in L^2(\mathcal{X},\nu)$ there exists unique sequence $(A_n)_{n\geq 0}$ of forms $A_n\in \mathcal{L}^{(n)}_{HS}(\mathcal{H}),$ such that
\begin{equation}
\label{eq111}
f=\sum^{\infty}_{n=0} I^{\nu}_n A_n \ \mbox{in} \ L^2(\mathcal{X},\nu).
\end{equation}
Moreover, \eqref{eq111} holds if and only if $f\circ T^{-1}=\sum^{\infty}_{n=0} I^{\mu}_n A_n$ in $L^2(\mathcal{X},\mu).$

 \end{lemma}

 \begin{proof}[Proof of Lemma \ref{isomorphism}]
 According to the Souslin theorem \cite[th.6.8.6]{Bog}, there exists measurable $S:\mathcal{X}\to \mathcal{X}$ such that $S(T(x))=x,$ $x\in \mathcal{X}_0.$ Hence, $\widehat{T}$ is the isometry of $L^2(\mathcal{X},\mu)$ onto
$L^2(\mathcal{X},\nu)$. All other assertions immediately
follow from this fact.
 \end{proof}

Operators $I^{\nu}_n$ defined in Lemma \ref{isomorphism} produce the orthogonal expansion of $L^2(\mathcal{X},\nu).$ In the following theorem sufficient conditions are given when operators $I^{\nu}_n$ can be written in a more detailed form.

\begin{remark}[Notational]
The space $\mathcal{L}^{(n)}_{HS}(\mathcal{H})$ is canonically identified with the subspace of the space of all symmetric $(n-m)-$linear $\mathcal{L}^{(m)}_{HS}(\mathcal{H})-$valued Hilbert-Schmidt forms on $\mathcal{H}.$ Accordingly, $I^{\mu}_{n-m} A_n$ is the $\mu-$square integrable $\mathcal{L}^{(m)}_{HS}(\mathcal{H})-$valued function on $\mathcal{X}.$
\end{remark}

\begin{theorem}\label{thm1}

Assume that additionally to conditions 1)-3) of Lemma \ref{isomorphism} conditions

4)  for each $x\in \mathcal{X}_0$ $T(x)-x \in \mathcal{H};$

5) $\nu\ll \mu$\\
are satisfied.

Then for all $n\geq 0$ and $A_n\in \mathcal{L}^{(n)}_{HS}(\mathcal{H})$
\begin{equation}
\label{eq113}
I^{\nu}_nA_n(x)=\sum^n_{m=0} \binom{n}{m} I^{\mu}_{n-m} A_n(x)(T(x)-x,\ldots,T(x)-x), \ \nu-\mbox{a.e.}
\end{equation}
 \end{theorem}

\begin{proof}[Proof of Theorem \ref{thm1}]
Consider the case $A_n\in \mathcal{L}^{(n)}_{HS,fin}(\mathcal{H}),$ $A_n$ satisfies \eqref{eq1}.
The symmetry of $A_n$ implies the equality
$$
A_n(h,\ldots,h,u,\ldots,u)=\binom{n}{m}^{-1}\sum\limits_
{\begin{subarray}{l}
0\leq m_i\leq k_i,\\ \sum^d_{i=1}m_i=m  \end{subarray}} \prod^d_{i=1} \Bigg(\binom{k_i}{m_i} (h,e_i)^{k_i-m_i}(u,e_i)^{m_i}\Bigg).
$$
Consequently,
$$
\sum\limits_
{\begin{subarray}{l}
0\leq m_i\leq k_i,\\ \sum^d_{i=1}m_i=m  \end{subarray}}
\prod^d_{i=1} \Bigg(\binom{k_i}{m_i} H_{k_i-m_i}(e_i(x))(T(x)-x,e_i)^{m_i}\Bigg)=
$$
$$
= \binom{n}{m} I^{\mu}_{n-m}A_n(x)(T(x)-x,\ldots,T(x)-x).
$$
On the other hand, the relation $H_n(x+y)=\sum^n_{m=0}\binom{n}{m} H_{n-m}(x)y^m$ implies following equalities.
$$
I^{\nu}_n A_n(x)=I^{\mu}_n A_n(T(x))=\prod^d_{i=1}H_{k_i}(e_i(x)+(T(x)-x,e_i))=
$$
$$
=\sum\limits_
{\begin{subarray}{l}
0\leq m_i\leq k_i,\\ 1\leq i \leq d \end{subarray}} \prod^d_{i=1} \Bigg(\binom{k_i}{m_i} H_{k_i-m_i}(e_i(x))(T(x)-x,e_i)^{m_i}\Bigg)=
$$
$$
=\sum^n_{m=0}\binom{n}{m} I^{\mu}_{n-m}A_n(x)(T(x)-x,\ldots,T(x)-x).
$$
The relation \eqref{eq113} is checked in the case $A_n\in \mathcal{L}^{(n)}_{HS,fin}(\mathcal{H}).$

Consider the general case $A_n\in \mathcal{L}^{(n)}_{HS}(\mathcal{H}).$ Choose a sequence $(A^{(k)}_n)_{k\geq 1}$ of the forms belonging to the span of $\mathcal{L}^{(n)}_{HS,fin}(\mathcal{H}),$ which converges to $A_n$ in $\mathcal{L}^{(n)}_{HS}(\mathcal{H}).$ Then for each $m=0,\ldots, n$ $(I^{\mu}_{n-m}A^{(k)}_n)_{k\geq 1}$ converges to
$I^{\mu}_{n-m}A_n$ in $L^2(\mathcal{X},\mu;\mathcal{L}^{(m)}_{HS}(\mathcal{H}))$ and $(\widehat{T}(I^{\mu}_nA^{(k)}_n))_{k\geq 1}$ converges to $\widehat{T}(I^{\mu}_nA_n)$ in $L^2(\mathcal{X},\nu).$
Passing to a subsequence we may asume that for $\mu-$a.a. $x\in \mathcal{X}$ for each $m=0,\ldots,n$
$I^{\mu}_{n-m} A^{(k)}_n(x)\rightarrow I^{\mu}_{n-m} A_n(x),$ $k\to \infty$ in $\mathcal{L}^{(m)}_{HS}(\mathcal{H})$  and for $\nu-$a.a. $x\in \mathcal{X}$ $\widehat{T}(I^{\mu}_n A^{(k)}_n)(x)\rightarrow \widehat{T}(I^{\mu}_n A_n)(x),$ $k\to \infty.$ The assumption $\nu\ll \mu$ implies that for $\nu-$a.a. $x\in \mathcal{X}$
$$
\widehat{T}(I^{\mu}_n A_n)(x)=\lim_{k\to \infty}\widehat{T}(I^{\mu}_n A^{(k)}_n)(x)=
$$
$$
=\lim_{k\to \infty} \sum^n_{m=0}\binom{n}{m} I^{\mu}_{n-m}A^{(k)}_n(x)(T(x)-x,\ldots,T(x)-x)=
$$
$$
=\sum^n_{m=0}\binom{n}{m} I^{\mu}_{n-m}A_n(x)(T(x)-x,\ldots,T(x)-x).
$$
\end{proof}

\section{Transformations of the Wiener measure}

In this section we study measurable transformations on the classical Wiener space $(\mathcal{C},\mathcal{H},\mu)$. In such setting for any probability measure $\nu$ on $(\mathcal{C},\cB)$ which is absolutely continuous with respect to the Wiener measure $\mu,$ the Girsanov theorem allows to construct the mapping $T$ satisfying conditions 1),2),4) of Theorem \ref{thm1}.

Denote by $\rho$ the density of $\nu$ with respect to $\mu,$ $\rho=\frac{d\nu}{d\mu}.$ By the Clark theorem \cite[th. 5.6]{3} there exists an adapted process $h:[0,1]\times \mathcal{C}\to \mbR,$ such that
$$
\rho(x)=1 + \int^1_0 h(s,x) dw_s,  \ \int^1_0 h^2(s)ds <\infty \ \mu-\mbox{a.s}.
$$
Denote $\rho(t,x)=1 + \int^t_0 h(s,x) dw_s,$ so that $(\rho(t))_{t\in[0,1]}$ is the continuous version of the martingale
$(\mathbb{E}[\rho/\cB_t])_{t\in[0,1]}.$ According to the Girsanov theorem \cite[th. 6.2]{3} the mapping
\begin{equation}
\label{eq31}
T(x)(t)=x(t)-\int^t_0 \frac{h(s,x)}{\rho(s,x)} ds
\end{equation}
sends measure $\nu$ to the measure $\mu.$ It is well-defined on the set
$$
\mathcal{X}_0=\{x\in \mathcal{X}: \inf_{t\in[0,1]} \rho(t,x)>0\}(=\{x\in \mathcal{X}: \rho(x)>0\})
$$
of full measure $\nu$ and evidently satisfy condition $T(x)-x\in \mathcal{H}.$
Unfortunately, $T$ may fail to be injective on the set of a full measure $\nu$ as it is seen in the famous example due to B.~S.~Tsirelson \cite{T}.

\begin{example} There exists bounded adapted function $b:[0,1]\times \mathcal{C} \to \mbR$ such that for any weak solution $(\xi_t,\tilde{w}_t)_{t\in[0,1]}$ of SDE
\begin{equation}
\label{eq32}
\begin{cases}
d\xi_t=b(t,\xi)dt+d\tilde{w}_t \\
\xi_0=0
\end{cases}
\end{equation}
the $\sigma-$field $\sigma(\xi)$ contains a proper set independent of $\tilde{w}$ \cite[ch. V, th. 18.3]{RW2}. Denote by $\nu$ the distribution of $\xi$ in $(\mathcal{C},\cB).$ Assume that the mapping $T$ constructed in \eqref{eq31} is injective on a set of full measure $\nu.$ Equivalent statement is that $\sigma(\xi)$ coincides with $\cB$ up to $\nu-$null sets. Hence, if we consider probability space $(\mathcal{C},\cB,\nu)$ and denote $\xi_t(x)=x(t),$ $\tilde{w}_t(x)=T(x)(t)$ we obtain weak solution $(\xi_t,\tilde{w}_t)_{t\in[0,1]}$ of \eqref{eq32} with $\sigma(\xi)=\sigma(\tilde{w}),$ which is a contradiction.

\end{example}

Next we turn to positive results. Let $g:[0,1]\to \mbR,$ $g(0)>0$ be continuous function;
$\tau_g(x)=\inf\{t\in [0,1]: x(t)=g(t)\}$ be the first moment when trajectory $x\in \mathcal{C}$ reaches $g$ ($\inf \emptyset = 1$).  Denote by $\vk_g$ the probability on $(\mathcal{C},\cB)$ defined via the density
$\frac{d\vk_g}{d\mu}=\mu(\tau_g=1)^{-1}\1_{\tau_g=1}$ with respect to the Wiener measure $\mu.$
In what follows we prove that the mapping $T_g$ defined in \eqref{eq31} for measure $\vk_g$ satisfies all the conditions of Theorem \ref{thm1}. Subsequently, this result is used to construct the orthogonal expansion of the space
$L^2(\mathcal{C},\vk_g).$ Further we will formulate and use analogous results for the case $t<1,$ i.e. for the space $(\mathcal{C}_t,\cB_t,\mu_t)$ and the measure $\vk_{g,t},$ $\frac{d\vk_{g,t}}{d\mu_t}(x)=\mu_t(\tau_g\geq t)^{-1}\1_{\tau_g \geq t}.$ Proofs will be omitted as they repeat proofs in the case $t=1.$

Denote by $\Gamma(g)$ the subgraph of $g:$ $\Gamma(g)=\{(s,y): 0\leq s \leq 1, \ y<g(s)\}.$ Let $\alpha(s,y,g)$ be the probability that the Brownian motion starting at the moment $s$ from the point $y$ doesn't reach $g:$
$$
\alpha(s,y,g)=\mathbb{P}(\forall t\in [s,1] \ y+w_{t-s}<g(t)), \ (s,y)\in \Gamma(g).
$$
In following lemmas the needed smoothness properties of the function $\alpha$ are proved.
\begin{lemma} \label{lem_cont_1}
The function $(s,y)\to \alpha(s,y,g),$ $(s,y)\in \Gamma(g)$ is continuous.
\end{lemma}
\begin{proof} [Proof of Lemma \ref{lem_cont_1}]
Assume that $(s_n,y_n)\in \Gamma(g),$ $(s_n,y_n)\to (s,y)\in \Gamma(g),$ $n\to \infty.$
Evidently,
$$
\overline{\lim}_{n\to\infty}(\{\forall t \in [s,1] \ y+w_t-w_s<g(t)\} \setminus
\{\forall t \in [s_n,1] \ y_n+w_t-w_{s_n}<g(t)\})=\emptyset,
$$
and
$$
\overline{\lim}_{n\to\infty}(\{\forall t \in [s_n,1] \ y_n+w_t-w_{s_n}<g(t)\} \setminus
\{\forall t \in [s,1] \ y+w_t-w_{s}<g(t)\})\subset
$$
$$
\subset \{\max_{t\in [s,1]} (y+w_t-w_s-g(t))=0\}.
$$
Denote $f(t)=g(s)-g(t),$ $x=g(s)-y>0.$ Without loss of generality assume that $s=0.$ It is enough to show that
\begin{equation}
\label{lem11}
\mathbb{P}(\max_{[0,1]} (w+f)=x)=0.
\end{equation}
Note that for all $x>0$ $\mathbb{P}(\max_{[0,1]} (w+f)\leq x)>0.$ According to \cite[th. 4.2.2]{Bog_G}
the function $F(x)=\ln \mathbb{P}(\max_{[0,1]} (w+f)\leq x),$ $x>0$ is concave. Hence, it is continuous and equality \eqref{lem11} and the convergence $\alpha(s_n,y_n,g)\to \alpha(s,y,g),$ $n\to\infty$ are proved.
\end{proof}

In the next lemma a useful representation of the function $\alpha$ is derived.
Consider $(s,z)\in \Gamma(g).$ There exists such $c>0$ that
\begin{equation}
\label{eq3.1}
\forall t\in[s,1] \ z-c(t-s)<g(t).
\end{equation}
For $r\geq s$ denote by $\beta(s,r,z,c,g)$ the probability that the Brownian motion starting at the moment $r$ from the point
$z-c(r-s)$ doesn't reach $g:$
$$
\beta(s,r,z,c,g)=\mathbb{P}(\forall t\in [r,1] \ z-c(r-s)+w_{t-r}<g(t))
$$
(put $\beta(s,r,z,c,g)=1$ for $r>1$).

\begin{lemma} \label{lem2}
For any $y<z$ the following representation holds.
\begin{equation}
\label{eq_repr}
\alpha(s,y,g)=\int^\infty_0 \beta(s,s+t,z,c,g)\frac{z-y}{\sqrt{2\pi t^3}}
e^{-\frac{(z-y-ct)^2}{2t}}dt.
\end{equation}

\end{lemma}

\begin{proof} [Proof of Lemma \ref{lem2}]
Denote by  $\sigma$ the first moment when the Brownian motion starting at the moment $s$ from the point $y$ reaches the line $z-c(t-s):$
$$
\sigma=\inf\{t\in[s,1]: \ y+w_{t-s}= z-c(t-s)\}.
$$
Then the needed representation follows from the strong Markov property of the Brownian motion and the expression for the density of $\sigma$ \cite[ch.I, (9.2)]{RW}.
$$
\alpha(s,y,g)=\mathbb{E} \beta(s,\sigma,z,c,g)=
$$
$$
=\int^\infty_0 \beta(s,s+t,z,c,g)\frac{z-y}{\sqrt{2\pi t^3}}
e^{-\frac{(z-y-ct)^2}{2t}}dt.
$$
\end{proof}

\begin{corollary} \label{cor1}
For each $s\in[0,1]$ the function $y\mapsto \alpha(s,y,g)$ is infinitely differentiable on
$(-\infty,g(s)).$ For every $n\geq 0$ the mapping
$$
(s,y)\to \frac{\partial^n \alpha}{\partial y^n} (s,y,g), \ (s,y)\in \Gamma(g)
$$
is continuous. In particular, for each compact $A\subset \Gamma(g)$ and $n\geq 0$
$$
\inf_{(s,y)\in A} \alpha(s,y,g)>0, \ \sup_{(s,y)\in A} \Bigg|\frac{\partial^n \alpha }{\partial y^n}(s,y,g)\Bigg|<\infty.
$$
\end{corollary}

\begin{proof}[Proof of Corollary \ref{cor1}]
The differentiability of $\alpha(s,y,g)$ in $y$ is easily seen from the representation \eqref{eq_repr}. Assume that $(s_k,y_k)\in\Gamma(g),$ $k\geq 0,$ $(s_k,y_k)\to (s_0,y_0), \ k\to \infty.$ Choose $z\in (y_0,g(s_0))$ and $c>0$ such that \eqref{eq3.1} holds. We may assume that for all $k$ $z\in(y_k,g(s_k))$ and $z-c(t-s_k)<g(t),$ $t\in[s_k,1].$
According to \eqref{eq_repr}, for all $k\geq 0$
$$
\frac{\partial^n \alpha}{\partial y^n} (s_k,y_k,g)=\int^\infty_0 \beta(s_k,s_k+t,z,c,g)
\frac{\partial^n}{\partial y^n}\Bigg(\frac{z-y}{\sqrt{2\pi t^3}}
e^{-\frac{(z-y-ct)^2}{2t}}\Bigg)\Bigg|_{y=y_k}dt
$$
and only the convergence $\beta(s_k,s_k+t,z,c,g)\to\beta(s_0,s_0+t,z,c,g), \ k\to \infty$ must be checked. The case when $s_k+t\geq 1$ infinitely often is obvious. In the case when $s_k+t\leq 1$ one has
$$
\beta(s_k,s_k+t,z,c,g)=\alpha(s_k+t,z-ct,g)
$$
and convergence follows from Lemma \ref{lem_cont_1}.
\end{proof}

\begin{corollary} \label{cor2}
Consider the sequence of continuous functions $g_k\in C([0,1]),$ $k\geq 1$ which is increasing to the continuous function $g\in C([0,1]).$ Then for each
$(s,y)\in \Gamma(g)$ and $n\geq 0$
$$
\frac{\partial^n \alpha}{\partial y^n} (s,y,g_k)\to \frac{\partial^n \alpha}{\partial y^n} (s,y,g), \ k\to \infty.
$$
\end{corollary}

\begin{proof}[Proof of Corollary \ref{cor2}]
Assume that $(s,y)\in \Gamma(g_1).$ Consider $z\in (y,g_1(s))$ and $c>0$ such that \eqref{eq3.1} is satisfied for all functions $g_k$. The representation from Lemma \ref{lem2} implies that only the convergence
$$
\beta(s,r,z,c,g_k)\to \beta(s,r,z,c,g), \ k\to \infty
$$
must be checked. It easily follows from the continuity of probability  measure \cite[L.1.14]{K}.
\end{proof}

In the next theorem the Clark representation for $\1_{\tau_g=1}$ is found and the expression \eqref{eq31} in the case of the measure $\vk_g$ is clarified. As a corollary the measurable isomorphism between spaces $(\mathcal{C},\cB,\nu)$ and
$(\mathcal{C},\cB,\mu)$ satisfying all the conditions of Theorem \ref{thm1} is constructed.

\begin{theorem}
\label{thm2} i) The Clark representation for $\1_{\tau_g=1}$ has the form
\begin{equation}
\label{eq_ind}
\1_{\tau_g=1}=\alpha(0,0,g)+\int^1_0 \frac{\partial \alpha}{\partial y}(s,w_s,g)\1_{\tau_g\geq s} dw_s.
\end{equation}

ii) The set $\mathcal{X}_0=\{x\in \mathcal{C}: x<g \}$ and the mapping
$$
T_g(x)(t)=x(t)-\int^t_0 \frac{\partial \ln \alpha_g}{\partial y} (s,x(s),g) ds, \ x\in\mathcal{X}_0
$$
possess following properties:

1) $\vk_g(\mathcal{X}_0)=1;$

2) $T_g:\mathcal{X}_0\to \mathcal{C}$ is measurable and injective;

3) $\vk_g \circ T^{-1}_g=\mu;$

4) for each $x\in \mathcal{X}_0$ $T_g(x)-x \in \mathcal{H}$.
\end{theorem}

\begin{proof} [Proof of Theorem \ref{thm2}]
i) Suppose that $g$ is continuously differentiable.
For every $s\in [0,1]$ and $y<g(0)$ let $\gamma(s,y)$ be the probability that the process $(y+w_t-w_s-(g(t)-g(s)))_{t\in[s,1]}$ doesn't achieve the level $g(0):$
$$
\gamma(s,y)=\mathbb{P}(\forall t\in [s,1] \ y+w_t-w_s-(g(t)-g(s))<g(0)).
$$
$\gamma$ satisfies the following boundary value problem \cite[ch.6, \S 5]{F}.
$$
\frac{\partial \gamma}{\partial s}(s,y)+\frac{1}{2}\frac{\partial^2 \gamma}{\partial^2 y}(s,y)-g'(s)\frac{\partial \gamma}{\partial y}(s,y)=0,
$$
$$
\gamma(1,y)=1, \ \gamma(s,g(0))=0.
$$
$\tau_g$ coincides with the first moment when the process $\eta_t=w_t-(g(t)-g(0)), \ t \in[0,1]$ achieves level $g(0).$ According to the It\^o formula,
$$
\1_{\tau_g=1}=\alpha(0,0,g)+\int^1_0 \frac{\partial \gamma}{\partial y}(s,\eta_s)\1_{\tau_g\geq s} dw_s.
$$
It is easy to see that in the case $y<g(s)$ one has
$$
\frac{\partial \alpha}{\partial y}(s,y,g)=\frac{\partial \gamma}{\partial y}(s,y-(g(s)-g(0))).
$$
Hence, the formula \eqref{eq_ind} is proved for continuously differentiable function $g.$

\begin{remark}
Used approach was proposed in \cite{D} to treat the case $g=\mbox{const}.$ Also, in the case $g=\mbox{const}$ formula
\eqref{eq_ind} was proved in \cite{PT} via the machinery of BV-functions on the Wiener space.
\end{remark}

Assume next that $g$ is arbitrary continuous function. Consider the sequence $(g_k)_{k\geq 1}$ of continuously differentiable functions which is increasing to the function $g.$ By the Clark theorem there exists an adapted square integrable $h:[0,1]\times \mathcal{C}\to \mbR,$ such that
$$
\1_{\tau_g=1}=\alpha(0,0,g) + \int^1_0 h(s) dw_s.
$$
It is enough to show that for a.a. $s\in[0,1]$
$$
h(s)=\frac{\partial \alpha}{\partial y}(s,w_s,g)\1_{\tau_g\geq s}, \ \mu_s-\mbox{a.e.}
$$
Evidently, $\1_{\tau_{g_k}=1}\nearrow \1_{\tau_{g}=1},$ $k\to\infty$ $\mu-$a.s. and in $L^2(\mathcal{C},\mu).$ Corollary \ref{cor2} implies that $\alpha(0,0,g_k)\to \alpha(0,0,g),$ $k\to \infty.$ Hence,
$$
\int_{\mathcal{X}} \Bigg(\int^1_0 \frac{\partial \alpha}{\partial y}(s,w_s,g_k)\1_{\tau_{g_k}\geq s} dw_s -
\int^1_0 h(s) dw_s \Bigg)^2 \mu(dx)=
$$
$$
=\int^1_0 \int_{\mathcal{X}} \Bigg( \frac{\partial \alpha}{\partial y}(s,x(s),g_k)\1_{\tau_{g_k}(x)\geq s} -h(s,x)  \Bigg)^2 \mu_s(dx) ds \to 0, \ k\to\infty.
$$
Passing to subsequences, we may assume that for a.a. $s\in [0,1]$ and $\mu_s-$a.a. $x\in \mathcal{C}_s$
$$
\frac{\partial \alpha}{\partial y}(s,x(s),g_k)\1_{\tau_{g_k}(x)\geq s}\to h(s,x), \ k\to \infty
$$
and $x(s)\ne g(s).$ If $\tau_g(x)<s$ then for all $k$ $\tau_{g_k}(x)<s$ and $h(s,x)=0.$ If $\tau_{g}(x)\geq s,$ then $x(s)<g(s)$ and corollary \ref{cor2} implies that $h(s,x)=\frac{\partial \alpha}{\partial y}(s,x(s),g).$
Formula \eqref{eq_ind} is proved for all $g\in C([0,1]).$

ii)  According to the Markov property of the Brownian motion
\begin{equation}
\label{eq_ind2}
\mathbb{E}[\1_{\tau_g=1}/\cB_s]=\1_{\tau_g\geq s}\alpha(s,w_s,g).
\end{equation}
From \eqref{eq_ind} and \eqref{eq_ind2} it follows that in the case of the measure $\vk_g$ the relation \eqref{eq31}  takes the form
$$
T_g(x)(t)=x(t)-\int^t_0 \frac{\partial \ln \alpha_g}{\partial y} (s,x(s),g) ds.
$$
Hence, as it is shown in the beginning of this chapter, the properties 1), 3), 4) are satisfied. We will check injectivity of the mapping $T_g.$ Assume that $x_1,x_2\in \mathcal{X}_0,$ $T_g(x_1)=T_g(x_2).$ There exist $a,\gamma>0$ such that for any $s\in [0,1]$ $x_1(s),x_2(s)\in [-a,g(s)-\gamma].$ Corollary \ref{cor1} implies the existence of a constant $A>0$ such that for all $s\in [0,1]$ $y\in [-a,g(s)-\gamma]$
$\Bigg|\frac{\partial^2 (\ln \alpha)}{\partial y^2}(s,y,g) \Bigg|\leq A.$ Hence, for each $t\in[0,1]$
$$
|x_1(t)-x_2(t)|\leq A \int^t_0 |x_1(s)-x_2(s)|ds
$$
which is possible only when $x_1=x_2.$

\end{proof}

As a corollary of Theorems \ref{thm1} and \ref{thm2} we construct the orthogonal expansion of the space $L^2(\mathcal{C}_t,\vk_{g,t}).$
For each $s\in [0,t],$ $y<g(s)$ denote by
$\alpha^t(s,y,g)$ the probability that the Brownian motion starting at the moment $s$ from the point $y$ doesn't reach $g$ up to the moment $t:$
$$
\alpha^t(s,y,g)=\mathbb{P}(\forall r\in [s,t] \ y+w_{r-s}<g(r)).
$$
Define the mapping $I^{\vk_{g,t}}_n:L^2_s([0,t]^n) \rightarrow L^2(\mathcal{C}_t,\vk_{g,t})$ as follows:
\begin{equation}
\label{comp}
I^{\vk_{g,t}}_na_n=\sum^n_{m=0} (-1)^m C^m_n \int_{[0,t]^m}\int_{[0,t]^{n-m}} a_n(r,s)dw_r \prod^m_{i=1}
\frac{\partial (\ln \alpha^t)}{\partial y}(s_i,w_{s_i},g)ds.
\end{equation}

\begin{corollary} \label{cor3}
Mappings $I^{\vk_{g,t}}_n$ possess following properties:

1) $\frac{1}{\sqrt{n!}}I^{\vk_{g,t}}_n$ is the isometrical embedding;

2) spaces $\cK_n(\mathcal{X}_t,\vk_{g,t})=I^{\vk_{g,t}}_n(L^2_s([0,t]^n))$ are mutually orthogonal;

3) $L^2(\mathcal{X}_t,\vk_{g,t})=\mathop{\oplus}\limits^\infty_{n=0}\cK_n(\mathcal{X}_t,\vk_{g,t}).$
\end{corollary}

\begin{remark}
According to the stochastic Fubini theorem \cite[th.4.18]{DaPZ}, the expression \eqref{comp} can be written in the form
$$
I^{\vk_{g,t}}_na_n=n!\sum_{1\leq k_1<\ldots  <k_m\leq n,0\leq m\leq n} (-1)^m \int^t_0 \int^{r_n}_0 \ldots \int^{r_2}_0
$$
$$
a_n(r_1,\ldots,r_n) dw_{r_1}\ldots dw_{r_{k_1-1}} \frac{\partial (\ln \alpha^t)}{\partial y}(r_{k_1},w_{r_{k_1}},g)dr_{k_1} dw_{r_{k_1}+1}\ldots dw_{r_n}.
$$
Formally this is equivalent to
$$
I^{\vk_{g,t}}_na_n=n!\int^t_0 \int^{r_n}_0 \ldots \int^{r_2}_0 a_n(r_1,\ldots,r_n)
$$
$$
(dw_{r_1}-\frac{\partial(\ln \alpha^t)}{\partial y}(r_1,w_{r_1},g)dr_1)\ldots
(dw_{r_n}-\frac{\partial(\ln \alpha^t)}{\partial y}(r_n,w_{r_n},g)dr_n).
$$
From the last expression the orthogonality of stochastic integrals $I^{\vk_{g,t}}_na_n$ and $I^{\vk_{g,t}}_ka_k$ of different degrees with respect to $\vk_{g,t}$ is obvious. Indeed, under the measure $\vk_{g,t}$ the process $(x(s)-\int^s_0 \frac{\partial(\ln \alpha^t)}{\partial y}(r,x(r),g)dr)_{s\in[0,t]}$ is a Brownian motion.
\end{remark}

\section{Orthogonal expansion of the space $L^2(\eta_g)$}

Recall that $\eta_g(\cdot)=w(\min(\tau_g(w),\cdot)),$ where $\tau_g(x)=\inf\{t\in [0,1]: x(t)=g(t)\}.$
Denote by $\nu_g$ the distribution of $\eta_g$ in the space $\mathcal{C}.$
$\nu_g$ is the image of the Wiener measure $\mu$ under the mapping
$F_g:\mathcal{C}\rightarrow \mathcal{C},$ $F_g(x)(t)=x(\min(\tau_g(x),t)).$ Denote by $\cB_{\tau_g}$ the $\sigma-$field of $\tau_g-$measurable sets:
$$
\cB_{\tau_g}=\{A\in \cB: \forall t\in[0,1] \ A\cap\{\tau_g\leq t\}\in \cB_t\}.
$$
In the next lemma the structure of the $\sigma-$field $\sigma(F_g)$ is described.
\begin{lemma}
\label{lem1}
\cite[L.1.3.3]{SW} Following $\sigma-$fields coincide:
$$
\{A\in \cB: F^{-1}_g(A)=A\}=\sigma(F_g)=\cB_{\tau_g}.
$$
\end{lemma}
According to Lemma \ref{lem1}
\begin{equation}
\label{eq5}
\mathbb{E}_{\mu}[I^{\mu}_n a_n/F_g]=n\int^{1}_0 \1_{\tau_g\geq t}I^{\mu_t}_{n-1} (a_n(\cdot,t))dw_t.
\end{equation}
In the next theorem it is proved that in order to get the orthogonal expansion of $L^2(\mathcal{C},\nu_g)$ it is enough to substitute $I^{\vk_{g,t}}_{n-1}$ for
$I^{\mu}_{n-1}$ in \eqref{eq5}.
Define the mapping $I^{\nu_g}_n$ as follows:
$$
I^{\nu_g}_n a_n= n\int^{\tau_g}_0 I^{\vk_{g,t}}_{n-1}(a_n(\cdot,t))dw_t, \ a_n\in L^2_s([0,1]^n).
$$

\begin{theorem}
\label{thm3}
Mappings $I^{\nu_g}_n$
possess following properties:

1) $I^{\nu_g}_n:L^2_s([0,1]^n) \rightarrow L^2(\mathcal{C},\nu_{g})$ is the isomorphic embedding;

2) spaces $\cK_n(\mathcal{C},\nu_g)=I^{\nu_g}_n(L^2_s([0,1]^n))$ are mutually orthogonal;

3) $L^2(\mathcal{C},\nu_g)=\mathop{\oplus}\limits^\infty_{n=0}\cK_n(\mathcal{C},\nu_g).$
\end{theorem}
\begin{proof} [Proof of Theorem \ref{thm3}]
Note that the function $x\rightarrow I^{\nu_g}_n a_n(x)$ is $\cB_{\tau_g}$-measurable. According to lemma \ref{lem1} it is $F_g$-invariant. Hence,
$$
\int_{\mathcal{C}} (I^{\nu_g}_n a_n(x))^2 \nu_g(dx)= \int_{\mathcal{C}} (I^{\nu_g}_n a_n(F_g(x)))^2 \mu(dx)=
$$
$$
=\int_{\mathcal{C}} (I^{\nu_g}_n a_n(x))^2 \mu(dx)=n^2\int_{\mathcal{C}} \Bigg(\int^1_0 \1_{\tau_g\geq t} I^{\vk_{g,t}}_{n-1} (a_n(\cdot,t))dw_t\Bigg)^2 d\mu=
$$
$$
=n^2\int^1_0 \int_{\mathcal{C}_t} \1_{\tau_g\geq t} (I^{\vk_{g,t}}_{n-1} (a_n(\cdot,t)))^2 d\mu_t dt=
$$
$$
=n^2\int^1_0 \mu(\tau_g\geq t) \int_{\mathcal{C}_t} (I^{\vk_{g,t}}_{n-1} (a_n(\cdot,t)))^2 d\vk_{g,t}dt=
$$
$$
=n n!\int^1_0 \mu(\tau_g\geq t) \int_{[0,t]^{n-1}} a_n(s,t)^2 ds dt.
$$
Consequently,
$$
\mu(\tau_g=1)n! \int_{[0,1]^n} a_n(t)^2dt\leq \int_{\mathcal{C}} (I^{\nu_g}_n a_n(x))^2 \nu_g(dx)\leq
$$
$$
\leq n! \int_{[0,1]^n} a_n(t)^2dt.
$$
This proves that $I^{\nu_g}_n$ is the isomorphic embedding.

Analogous calculation proves the orthogonality of spaces $\cK_n(\mathcal{C},\nu_g).$ The totality of spaces $\cK_n(\mathcal{C},\nu_g)$ in $L^2(\mathcal{C},\nu_g)$ is left to be proved. Consider $f\in L^2(\mathcal{C},\nu_g).$ Without loss of generality we may assume that $\int_\mathcal{C} fd\nu_g=0.$ The composition $f\circ F_g$ belongs to $L^2(\mathcal{X},\mu).$ Hence, there exists an adapted square integrable function $h:[0,1]\times \mathcal{C}\rightarrow \mbR,$ such that
$$
f\circ F_g=\int^1_0 h_t dw_t, \mu-a.e.
$$
Lemma \ref{lem1} implies that after conditioning with respect to $\sigma(F_g)$ one has
$$
f\circ F_g=\int^1_0 \1_{\tau_g\geq t}h_t dw_t, \mu-a.e.
$$
Both sides of the last equality are measurable with respect to $F_g.$ Hence,
$$
f=\int^1_0 \1_{\tau_g\geq t}h_t dw_t, \nu_g-a.e.
$$
For a.a. $t\in [0,1]$ $h_t\in L^2(\mathcal{C}_t, \mu_t)\subset L^2(\mathcal{C}_t, \vk_{g,t}).$ Corollary \ref{cor3} implies the existence of unique sequence $a^t_n\in L^2_s([0, t]^n), \ n\geq 0,$ such that
$$
h_t=\sum^\infty_{n=0} I^{\vk_{g,t}}_n a^t_n.
$$
Hence, the square of norm of $f$ can be calculated in the following way.
$$
\int_{\mathcal{C}} f^2 d\nu_g=
\int_{\mathcal{C}}\Bigg(\int^1_0 \1_{\tau_g\geq t}h_t dw_t\Bigg)^2 d\mu=
\int^1_0 \int_{\mathcal{C}_t} \1_{\tau_g\geq t} h_t^2 d\mu_t dt=
$$
$$
=
\int^1_0 \mu(\tau_g\geq t)\int_{\mathcal{X}_t} (h(t,x))^2 \vk_{g,t}(dx) dt=
$$
$$
=\int^1_0 \mu(\tau_g\geq t)\sum^{\infty}_{n=0} \int_{\mathcal{C}_t} (I^{\vk_{g,t}}_n a^t_n)^2 d\vk_{g,t} dt=
$$
$$
=
\sum^{\infty}_{n=0} \int^1_0 \mu(\tau_g\geq t) \int_{\mathcal{C}_t} (I^{\vk_{g,t}}_n a^t_n)^2 d\vk_{g,t} dt=
\sum^{\infty}_{n=0} \int_{\mathcal{C}} (I^{\nu_g}_{n+1} b_{n+1})^2 d\nu_g.
$$
In the last expression $b_{n+1}$ denotes the symmetrization of the function $(s,t)\rightarrow a^t_n(s)$ to $[0,1]^{n+1}.$ This proves the Parseval identity for $f.$ Theorem is proved.
\end{proof}

\begin{example} Assume that $f\in L^2(\eta_g)$ has the Clark representation of the form
$$
f=\int^{\tau_g}_0 f_t(\eta_g(t))d\eta_g(t).
$$
As is seen from the proof of Theorem \ref{thm3} to expand $f$ as the series $\sum^{\infty}_{n=1} I^{\nu_g}_n b_n$ one has to obtain the representation of the functional $x\to f_t(x(t))$ with respect to the measure $\vk_{g,t}$
$$
f_t(x(t))=\sum^{\infty}_{n=0} I^{\vk_{g,t}}_n a^t_n
$$
and define $b_n$ as the symmetrization of $(s,t)\to a^t_{n-1}(s).$
According to the Lemma \ref{isomorphism} kernels $a^t_n$ coincide with kernels in the It\^o-Wiener expansion of $f_t(\xi_t),$ where $(\xi_s)_{s\in[0,t]}$ is the solution of the SDE
\begin{equation}
\begin{cases}
d\xi_s = \frac{\partial \ln \alpha_g}{\partial y} (s,\xi_s,g) ds + dw_s \\
\xi_0=0.
\end{cases}
\end{equation}
The latter expansion is given by the Krylov-Veretennikov formula \eqref{intro2}.
\end{example}

\end{document}